\documentclass[a4paper,10pt]{amsart}

\usepackage{hyperref,amssymb,enumerate,url}

\title{Reverse mathematics and initial intervals}
\author{Emanuele Frittaion}
   \address{Dipartimento di Matematica e Informatica,
   Universit\`{a} di Udine,
   33100 Udine,
   Italy}
\email{emanuele.frittaion@uniud.it}

\author{Alberto Marcone}
   \address{Dipartimento di Matematica e Informatica,
   Universit\`{a} di Udine,
   33100 Udine,
   Italy}
\email{alberto.marcone@uniud.it}

\date{October 28, 2013}
\thanks{We thank Gregory Igusa, Matthew Hendtlass and Henry Towsner for
useful discussions about the subject.
In particular Igusa suggested the details of the proof of Lemma \ref{REC}.
We also thank the anonymous referees for a careful reading of the paper and several useful
suggestions.}

\newcommand{\N}{\mathbb{N}}

\newcommand{\Q}{\mathbb{Q}}

\newcommand{\KB}{\mathrm{KB}}

\newcommand{\down}[1]{P_{\preceq {#1}}}  % Downwards closure
\newcommand{\up}[1]{P_{\succeq {#1}}}    % Upwards closure
\newcommand{\downs}[1]{P_{\prec {#1}}}   % strict downwards closure
\newcommand{\ups}[1]{P_{\succ {#1}}}     % strict upwards closure
\newcommand{\inc}[1]{P_{\perp {#1}}}     % strict upwards closure
\DeclareMathOperator{\Down}{\downarrow}  % Downwards closure of a set
\DeclareMathOperator{\restr}{\restriction} % Restriction
\DeclareMathOperator{\Int}{\mathcal{I}}    % Set of initial intervals
\DeclareMathOperator{\Fin}{Fin}

\newcommand{\conc}{{{}^\smallfrown}}

\newcommand{\PI}[2]{\ensuremath{\boldsymbol\Pi^{#1}_{#2}}}
\newcommand{\SI}[2]{\ensuremath{\boldsymbol\Sigma^{#1}_{#2}}}
\newcommand{\DE}[2]{\ensuremath{\boldsymbol\Delta^{#1}_{#2}}}

\newcommand{\RCA}{{\ensuremath{\mathsf{RCA}_0}}}
\newcommand{\WWKL}{{\ensuremath{\mathsf{WWKL}_0}}}
\newcommand{\WKL}{{\ensuremath{\mathsf{WKL}_0}}}

\newcommand{\ACA}{{\ensuremath{\mathsf{ACA}_0}}}
\newcommand{\ATR}{{\ensuremath{\mathsf{ATR}_0}}}
\newcommand{\oneone}{\PI11-{\ensuremath{\mathsf{CA}_0}}}
\newcommand{\AC}{\SI11-\ensuremath{\mathsf{AC}_0}}
\newcommand{\DC}{\SI11-\ensuremath{\mathsf{DC}_0}}
\newcommand{\ATRl}[1]{\ensuremath{\mathrm{ATR}_0^{#1}}}
\newcommand{\NCF}{\ensuremath{\mathsf{NCF}}}

\newcommand{\REC}{\ensuremath{\mathbf{REC}}}

\newcommand{\si}[1]{\ensuremath{\sigma \conc \langle {#1} \rangle}}

\theoremstyle{plain}
\newtheorem{theorem}{Theorem}[section]
\newtheorem*{claim}{Claim}
\newtheorem{nclaim}{Claim}
\newtheorem{lemma}[theorem]{Lemma}
\newtheorem{corollary}[theorem]{Corollary}

\theoremstyle{definition}
\newtheorem{definition}[theorem]{Definition}

\begin{document}

\begin{abstract}
In this paper we study the reverse mathematics of two theorems by Bonnet
about partial orders. These results concern the structure and cardinality
of the collection of the initial intervals. The first theorem states that a
partial order has no infinite antichains if and only if its initial
intervals are finite unions of ideals. The second one asserts that a
countable partial order is scattered and does not contain infinite
antichains if and only if it has countably many initial intervals. We show
that the left to right directions of these theorems are equivalent to \ACA\
and \ATR, respectively. On the other hand, the opposite directions are both
provable in \WKL, but not in \RCA. We also prove the equivalence with \ACA\
of the following result of Erd\"{o}s and Tarski: a partial order with no
infinite strong antichains has no arbitrarily large finite strong
antichains.
\end{abstract}

\maketitle

\tableofcontents

%%%%%%%%%%%%%%%%%%%%%%%%%%%%%%%%%%%%%%%%%%%%%%%%%%%%%%%%%%%%%%%%%%%%%%%%%%%%%%%%

\section{Introduction}

In this paper we study from the viewpoint of reverse mathematics some
theorems dealing with the structure and the cardinality of the collection of
initial intervals (also called downward closed subsets) in a partial order.
Recall that an ideal is an initial interval such that every pair of elements
is compatible (i.e.\ has a common upper bound) in the interval.

The first result is a characterization of partial orders with no infinite
antichains in terms of the decomposition of initial intervals into union of
ideals. It is due to Bonnet \cite[Lemma 2]{Bon73} and can be found in
Fra\"iss\'{e}'s monograph \cite[\S4.7.2]{Fra00}:

\begin{theorem}\label{Bonnet1}
A partial order has no infinite antichains if and only if every initial
interval is a finite union of ideals.
\end{theorem}

In \cite{PS} Theorem \ref{Bonnet1} is attributed to Erd\"{o}s and Tarski because
its \lq hard\rq\ (left to right) direction can be deduced quite easily from
the following result, which is part of \cite[Theorem 1]{ErdTar43}:

\begin{theorem}\label{ET1}
If a partial order has no infinite strong antichains then it has no
arbitrarily large finite strong antichains.
\end{theorem}

Here, by strong antichain we mean a set of pairwise incompatible (and not
only incomparable, as in antichain) elements. (Notice that Erd\"{o}s and Tarski
work with what we would call filters and final intervals.)

An intermediate step between Theorems \ref{ET1} and \ref{Bonnet1} is the
following characterization of partial orders with no infinite strong
antichains:

\begin{theorem}\label{ET2}
A partial order has no infinite strong antichains if and only if it is a
finite union of ideals.
\end{theorem}

Our proof of Lemma \ref{lemma 2} shows how to deduce the left to right
direction of Theorem \ref{ET2} from Theorem \ref{ET1}.

In \cite{Bon73} Theorem \ref{Bonnet1} is a step in the proof of the following
result, which is also featured in Fra\"{\i}ss\'{e}'s monograph
\cite[\S6.7]{Fra00}:

\begin{theorem}\label{Bonnet2}
If an infinite partial order $P$ is scattered (i.e.\ there is no embedding of
the rationals into $P$) and has no infinite antichains, then the set of
initial intervals of $P$ has the same cardinality of $P$.
\end{theorem}

The converse of Theorem \ref{Bonnet2} is in general false, but it holds when
$|P|<2^{\aleph_0}$, and in particular when $P$ is countable:

\begin{theorem}\label{Bonnet3}
A countable partial order is scattered and has no infinite antichains if and
only if it has countably many initial intervals.
\end{theorem}\bigskip

The program of reverse mathematics (\cite{sosoa} is the basic reference)
gauges the strength of mathematical theorems by means of the subsystems of
second order arithmetic necessary for their proofs. This approach allows only
the study of statements about countable (or countably coded) objects. We
therefore study the strength of Theorem \ref{Bonnet3} and of the
restrictions of Theorems \ref{Bonnet1}, \ref{ET1} and \ref{ET2} to countable
partial orders. We notice that \cite{ErdTar43,Bon73,Fra00} put no restriction
on the cardinality of the partial order and therefore often use set-theoretic
techniques which are not available in (subsystems of) second order
arithmetic. On the other hand we can always assume that the partial orders
are defined on a subset of the set of the natural numbers, and this is on
occasion helpful.

Since Theorems \ref{Bonnet1}, \ref{ET2}, and \ref{Bonnet3} are equivalences,
we study separately the two implications, which turn out to have different
axiomatic strengths. In particular, the \lq easy\rq\ (right to left)
directions of Theorems \ref{Bonnet1} and \ref{Bonnet3} are quite interesting
from the viewpoint of reverse mathematics and we are not able to settle the
problem of establishing their strength, leaving open the possibility that
they have strength intermediate between \RCA\ and \WKL.\medskip

We assume familiarity with the \lq big five\rq\ of reverse mathematics,
namely, in order of increasing strength, \RCA, \WKL, \ACA, \ATR, and
\oneone.\medskip

We now state our main results and at the same time describe the organization
of the paper. In section \ref{section:preliminaries} we establish our
notation and terminology and recall some basic results. In section
\ref{section:iis&eu} we prove a couple of technical lemmas that are useful
later on.\smallskip

In Section \ref{section:->} we consider Theorem \ref{ET1} and the left to
right directions of Theorems \ref{Bonnet1}, \ref{ET2}, and \ref{Bonnet3}.
Subsection \ref{section:ACA} culminates in Theorem \ref{equivACA} where we
prove, over \RCA, the equivalence of \ACA\ with each of the three statements:
\begin{itemize}
 \item in a countable partial order with no infinite antichains every
     initial interval is a finite union of ideals;
 \item in a countable partial order with no infinite strong antichains
     there is a bound on the size of the strong antichains;
 \item every countable partial order with no infinite strong antichains is
     a finite union of ideals.
\end{itemize}

In subsection \ref{section:uncountably} we show that the statement
\begin{itemize}
 \item every countable partial order which is scattered and has no infinite
     antichains has countably many initial intervals.
\end{itemize}
is equivalent to \ATR\ over \ACA\ (Theorem \ref{Bonnet3->ATR}). To obtain the
reversal we slightly modify a proof in \cite{Clo89}.\smallskip

In section \ref{section:<-} we deal with the right to left directions of
Theorems \ref{Bonnet1}, \ref{ET2}, and \ref{Bonnet3}, i.e.\ with the
statements:
\begin{itemize}
 \item if every initial interval of a countable partial order is a finite
     union of ideals, then the partial order has no infinite antichains;
 \item if a countable partial order is a finite union of ideals then it has
     no infinite strong antichains;
 \item if a countable partial order has countably many initial intervals,
     then it has no infinite antichains;
 \item if a countable partial order has countably many initial intervals,
     then it is scattered.
\end{itemize}
The obvious proofs of these statements go through in \ACA, but we show that
they are all provable in weaker systems. In fact \RCA\ proves the second and
fourth statement (Lemma \ref{ET2<-} and Theorem \ref{Bonnet3<-RCA}). On the
other hand, the first and third statement are both provable in \WKL\
(Theorems \ref{Bonnet1<-WKL} and \ref{Bonnet3<-WKL}) and fail in the
$\omega$-model of computable sets and hence cannot be proved in \RCA\
(Theorems \ref{Bonnet1<-nRCA} and \ref{Bonnet3<-nRCA}). Our results thus do
not completely determine the strength of these two statements.\smallskip

In Section \ref{section:problems} we briefly discuss the open problems left
by our results and mention some partial answers obtained by other authors
after a first draft of this paper was circulated.

\section{Terminology, notation, and basic facts}\label{section:preliminaries}

All definitions in this section are made in \RCA.

\subsection{Finite sequences and trees}
We typically use $\sigma$ and $\tau$ to denote finite sequences of natural
numbers, that is elements of $\N^{<\N}$. Often they belong to $2^{<\N}$,
i.e.\ they are binary, and in one occasion to $3^{<\N}$, i.e.\ they are
ternary. Let $|\sigma|$ be the length of $\sigma$ and list it as $\langle
\sigma(0), \dots, \sigma(|\sigma|-1) \rangle$. In particular $\langle
\rangle$ is the unique sequence of length $0$. We write $\sigma \sqsubseteq
\tau$ to mean that $\sigma$ is an initial segment of $\tau$, while $\sigma
\conc \tau$ denotes the concatenation of $\sigma$ and $\tau$. By $\sigma
\restr k$ we mean the initial segment of $\sigma$ of length $k$ and
similarly, when $f$ is a function, $f \restr k$ is the finite sequence
$\langle f(0), \dots, f(k-1) \rangle$.

A \emph{tree} $T$ is a set of finite sequences such that $\tau \in T$ and
$\sigma \sqsubseteq \tau$ imply $\sigma \in T$. A tree is \emph{pruned} if it
contains no endnodes, i.e.\ $(\forall \sigma \in T) (\exists \tau \in T)
\sigma \sqsubset \tau$. A \emph{path} in $T$ is a function $f$ such that for
all $n$ the finite sequence $f \restr n$ belongs to $T$. We write $[T]$ to
denote the collection of all paths in $T$: $[T]$ does not formally exists in
second order arithmetic, but $f \in [T]$ is a convenient shorthand.

A tree $T$ is \emph{perfect} if for all $\sigma \in T$ there exist $\tau_0,
\tau_1 \in T$ such that $\sigma \sqsubseteq \tau_0, \tau_1$ and neither
$\tau_0 \sqsubseteq \tau_1$ nor $\tau_1 \sqsubseteq \tau_0$ hold. A tree $T$
has \emph{countably many paths} if there exists a sequence $\{f_n \colon n
\in \N\}$ (coded by a single set) such that for every $f \in [T]$ we have
$f=f_n$ for some $n \in \N$. If $T$ does not have countably many paths then
we say that it has \emph{uncountably many paths}.

By \cite[Theorem V.5.5]{sosoa} \ATR\ is equivalent to the perfect tree
theorem:

\begin{theorem}[\ACA]\label{ptt}
The following are equivalent:
\begin{enumerate}
 \item \ATR;
 \item every tree with uncountably many paths contains a perfect subtree.
\end{enumerate}
\end{theorem}

\subsection{Partial orders}
Within \RCA\ saying that $(P, {\preceq})$ is a partial order means that $P
\subseteq \N$ and ${\preceq} \subseteq P \times P$ is reflexive,
antisymmetric and transitive. As usual, we use $\prec$ to denote the strict
order. From now on we refer to $(P,{\preceq})$ simply as $P$. When we deal
with several partial orders at the same time, we use subscripts as in
$\preceq_P$ to distinguish between the relations.

Finite partial orders can easily be studied in \RCA\ and hence, whenever it
is convenient and without further notice, we assume that $P$ is infinite.

Every time we define a partial order $\preceq$ on a set $P$ we assume
reflexivity, and focus on explaining when distinct elements are related and
on checking transitivity.

We say that $x,y \in P$ are \emph{comparable} if $x \preceq y$ or $y \preceq
x$. If $x$ and $y$ are incomparable we write $x \perp y$. A partial order $P$
is a \emph{linear order} if all its elements are pairwise comparable. A
linear order $P$ is \emph{dense} if for all $x,y\in P$ such that $x\prec y$
there exists $z\in P$ with $x \prec z \prec y$.

A subset $D \subseteq P$ is an \emph{antichain} if all its elements are
pairwise incomparable, i.e.
\[  (\forall x,y\in D)(x \neq y \implies x \perp y).\]

We say that $x,y \in P$ are \emph{compatible in} $P$ if there is $z \in P$
such that $x \preceq z$ and $y \preceq z$. Notice that two elements of $P$
might be compatible in $P$ but not in some $X \subseteq P$ to which they
belong.

A subset $S \subseteq P$ is a \emph{strong antichain in} $P$ if its elements
are pairwise incompatible in $P$, i.e.
\[   (\forall x,y \in S)(\forall z \in P)(x,y \preceq z \implies x=y).\]

A subset $I \subseteq P$ is an \emph{initial interval} of $P$ if
\[     (\forall x, y\in P)(x \preceq y \land y \in I \implies x \in I).\]

An initial interval $A$ of $P$ is an \emph{ideal} if every two elements of
$A$ are compatible in $A$, i.e.
\[ (\forall x,y \in A)(\exists z \in A)(x \preceq z \land y \preceq z).\]

If $x \in P$ we let $\inc x = \{y \in P \colon x \perp y\}$ and define the
\emph{upper and lower cones} determined by $x$ setting
\[ \up x = \{y \in P \colon x \preceq y\} \text{ and } \down x = \{y \in P
\colon y \preceq x\}.\] $\ups x$ and $\downs x$ are defined in the obvious
way. If $X \subseteq P$ we write $\Down X$ for the \emph{downward closure} of
$X$, i.e.\ $\bigcup_{x \in X} \down x$. Notice that the existence of $\Down
X$ as a set is equivalent to \ACA\ over \RCA.

\subsection{Well-partial orders, scattered partial orders and lexicographic
sums}
A partial order $P$ is \emph{well-founded} if $P$ contains no infinite
descending sequence, i.e.\ no function $f\colon \N \to P$ such that $f(i)
\succ f(j)$ for all $i<j$. A well-founded linear order is a
\emph{well-order}.

A partial order $P$ is a \emph{well-partial order} if for every function
$f\colon \N \to P$ there exist $i<j$ such that $f(i)\preceq f(j)$. There are
many classically equivalent definitions of well-partial order. In particular
a well-partial order is a well-founded partial order with no infinite
antichains. For a reverse mathematics study of these equivalences we refer to
\cite{ChoMarSol04}. For our purposes, it is enough to know that all these
equivalences are provable in \ACA\ and that \RCA\ suffices to show that every
well-partial order is well-founded and has no infinite antichains.

The \emph{Kleene-Brouwer order} on finite sequences is the linear order
defined by $\sigma \leq_\KB \tau$ if either $\tau \sqsubseteq \sigma$ or
$\sigma(i)< \tau(i)$ for the least $i$ such that $\sigma(i) \neq \tau(i)$.
One of the main features of $\leq_\KB$ is that, provably in \ACA, its
restriction to a tree $T$ is a well-order if and only if $T$ has no paths
(\cite[Lemma V.1.3]{sosoa}).

An \emph{embedding} of a partial order $Q$ into a partial order $P$ is a
function $f \colon Q \to P$ such that for all $x,y \in Q$ we have $x
\preceq_Q y$ if and only if $f(x) \preceq_P f(y)$. A partial order $P$ is
\emph{scattered} if there is no embedding of $\Q$ (the order of the
rationals) into $P$.

\begin{lemma}[\RCA]\label{scatt}
A partial order is scattered if and only if it does not contain any dense
linear order.
\end{lemma}
\begin{proof}
The left to right is immediate because \RCA\ suffices to carry out the usual
back-and-forth argument. For the other direction, given an embedding $f: \Q
\to P$ by recursion we can find $D \subseteq \Q$ dense such that $f$
restricted to $D$ is strictly increasing with respect to the ordering of the
natural numbers. Thus the range of $f$ restricted to $D$ exists in \RCA\ and
is a dense linear order.
\end{proof}

If $P$ is a partial order and $\{P_x \colon x \in P\}$ is a sequence of
partial orders indexed by $P$ we define the lexicographic sum of the $P_x$
along $P$, denoted by $\sum_{x \in P} P_x$, to be the partial order on the
set $Q = \{(x,y) \colon x \in P \land y \in P_x\}$ defined by
\[(x,y) \preceq_Q (x',y') \iff x \prec_P x' \lor (x=x' \land y \preceq_{P_x}
y').\]

\begin{lemma}[\RCA]\label{scattls}
The lexicographic sum of scattered partial orders along a scattered partial
order is scattered.
\end{lemma}
\begin{proof}
Let $Q = \sum_{x \in P} P_x$ be a lexicographic sum and suppose that $Q$ is
not scattered. Fix an embedding $f \colon \Q \to Q$.

First suppose that for some $a <_\Q b$ and $x \in P$ we have $f(a)=(x,y)$ and
$f(b)=(x,y')$. Then, the composition of $f$ with the projection on the second
coordinate is an embedding of the rational interval $(a,b)_\Q$ into $P_x$.
Since $\Q$ embeds into its open intervals, $P_x$ is not scattered.

Otherwise, composing $f$ with the projection on the first coordinate, we
obtain an embedding of $\Q$ into $P$, and $P$ is not scattered.
\end{proof}

\subsection{The set of initial intervals}
We denote by $\Int(P)$ the collection of initial intervals of the partial
order $P$. In second order arithmetic, $\Int(P)$ does not formally exist, and
$I \in \Int(P)$ is a shorthand for the formula \lq\lq $I$ is an initial
interval of $P$\rq\rq. To study Theorem \ref{Bonnet3} we need to discuss the
cardinality of $\Int(P)$.

We say that the partial order $P$ has \emph{countably many initial intervals}
if there exists a sequence $\{I_n \colon n \in \N\}$ such that for every $I
\in \Int(P)$ we have $I=I_n$ for some $n \in \N$. Otherwise, we say that $P$
has \emph{uncountably many initial intervals}.

Within \ACA\ we can prove that, if $P$ has countably many initial intervals,
then there exists a sequence $\{I_n \colon n \in \N\}$ such that $I \in
\Int(P)$ if and only if there exists $n \in \N$ such that $I=I_n$. In this
case we write $\Int(P) = \{I_n \colon n \in \N\}$.

The partial order $P$ has \emph{perfectly many initial intervals} if there
exists a nonempty perfect tree $T \subseteq 2^{<\N}$ such that $[T] \subseteq
\Int(P)$, that is, for all $f \in [T]$, the set $\{x \in \N \colon f(x)=1\}
\in \Int(P)$.

A useful tool for studying the notions we just defined is the \emph{tree of
finite approximations of initial intervals} of the partial order $P$. We
define the tree $T(P) \subseteq 2^{<\N}$ by letting $\sigma \in T(P)$ if and
only if for all $x,y< |\sigma|$:
\begin{itemize}
 \item $\sigma(x)=1$ implies $x \in P$;
 \item $\sigma(y)=1$ and $x \preceq y$ imply $\sigma(x)=1$.
\end{itemize}
Notice that $T(P)$ is a pruned tree and that the paths in $T(P)$ are exactly
the characteristic functions of the initial intervals of $P$. From the latter
observation we easily obtain:

\begin{lemma}[\RCA]\label{T(P)}
Let $P$ be a partial order.
\begin{enumerate}[\quad(i)]
 \item $P$ has countably many initial intervals if and only if $T(P)$ has
     countably many paths;
 \item $P$ has perfectly many initial intervals if and only if $T(P)$
     contains a perfect subtree.
\end{enumerate}
\end{lemma}

In particular, the formula \lq\lq$P$ has perfectly many initial
intervals\rq\rq\ is provably \SI11 within \RCA. Moreover a straightforward
diagonal argument shows in \RCA\ that a nonempty perfect tree has uncountably
many paths. Therefore we have that \RCA\ proves that a partial order with
perfectly many initial intervals has uncountably many initial intervals.
Using the perfect tree theorem we obtain that \ATR\ proves that a partial
order with uncountably many initial intervals has actually perfectly many
initial intervals. This implies that the formula \lq\lq$P$ has uncountably
many initial intervals\rq\rq\ is provably \SI11 within \ATR.

In connection with this recall the following result due to Peter Clote
\cite{Clo89}:

\begin{theorem}[\ACA]\label{countable or perfect}
The following are equivalent:
\begin{enumerate}
 \item \ATR;
 \item any linear order has countably many or perfectly many initial
     intervals;
 \item any scattered linear order has countably many initial intervals.
\end{enumerate}
\end{theorem}

Clote actually states the equivalence of \ATR\ only with (2), but his proofs
yield also the equivalence with (3).

\subsection{The system \ATRl X}

Recall that, by \cite[Theorem VIII.3.15]{sosoa}, \ATR\ is equivalent over
\ACA\ to the statement
\[ (\forall X)(\forall a \in \mathcal{O}^X) (H_a^X \text{ exists})\]
where $\mathcal{O}^X$ is the collection of (indices for) $X$-computable
ordinals and $H_a^X$ codes the iteration of the jump along $a$ starting from
$X$. This naturally leads to consider lightface versions of \ATR, as in
\cite{Tan89}, \cite{Tan90}, and \cite{Borelqo}. Here we make explicit mention
of the set parameter we use (rather then deal only with the parameterless
case and then invoke relativization) and let \ATRl X be \ACA\ plus the
formula $(\forall a \in \mathcal{O}^X) (H_a^{X} \text{ exists})$. In \ATRl X
one can prove arithmetical transfinite recursion along any $X$-computable
well-order.

By checking the proof of the forward direction of Theorem \ref{ptt} one
readily realizes that \ATRl X proves the perfect tree theorem for
$X$-computable trees:

\begin{theorem}[\ATRl X]\label{PTT lightface}
Every $X$-computable tree with uncountably many paths contains a perfect
subtree.
\end{theorem}

The following is \cite[Lemma VIII.4.19]{sosoa}:

\begin{theorem}[\ATRl X]\label{model existence theorem}
There exists a countable coded $\omega$-model $M$ such that $X \in M$ and $M$
satisfies \DC.
\end{theorem}

We will use the following corollary:

\begin{corollary}[\ATR]\label{model existence}
For all $X$ and $Y$ there exists a countable coded $\omega$-model $M$ such
that $X,Y \in M$ and $M$ satisfies both \DC\ and \ATRl{X}.
\end{corollary}
\begin{proof}
We argue in \ATR\ and let $X$ and $Y$ be given. By \AC, which is a
consequence of \ATR, the main axiom of \ATRl X is equivalent to a \SI11\
formula $(\exists Z) \varphi (Z,X)$ with $\varphi$ arithmetic. This formula
is true in \ATR, and hence we can fix $Z$ such that $\varphi (Z,X)$. By
Theorem \ref{model existence theorem} there exists a countable coded
$\omega$-model $M$ of \DC\ such that $X \oplus Y \oplus Z \in M$. In
particular, $X,Y \in M$ and, as $Z \in M$ and $M$ is a model of \DC\ (hence
also of \AC), $M$ satisfies \ATRl{X}.
\end{proof}

\section{Initial interval separation and essential unions}\label{section:iis&eu}

In this section we prove two technical results that are useful in the
remainder of the paper.

\subsection{Initial interval separation}
Our first result is a new equivalence with \WKL, inspired by the usual \SI01
separation (\cite[Lemma IV.4.4]{sosoa}) but producing separating sets which
are also initial intervals.

\begin{lemma}\label{iis}
Over \RCA, the following are equivalent:
\begin{enumerate}
 \item \WKL;
 \item \SI01 initial interval separation. Let $P$ be a partial order and
     $\varphi(x)$, $\psi(x)$ be \SI01 formulas with one distinguished free
   number variable.

   If $(\forall x,y \in P) (\varphi(x) \land \psi(y) \implies y \npreceq
   x)$, then there exists an initial interval $I$ of $P$ such that
\[  (\forall x \in P)((\varphi(x) \implies x \in I) \land (\psi(x) \implies x
\notin I)). \]
 \item Initial interval separation. Let $P$ be a partial order and suppose
     $A,B \subseteq P$ are such that $(\forall x \in A) (\forall y \in B) y
   \npreceq x$. Then there exists an initial interval $I$ of $P$ such that
   $A \subseteq I$ and $B \cap I = \emptyset$.
\end{enumerate}
\end{lemma}
\begin{proof}
We first assume \WKL\ and prove (2). Fix the partial order $P$ and let
$\varphi(x) \equiv (\exists m) \varphi_0 (x,m)$ and $\psi(n) \equiv (\exists
m) \psi_0(x,m)$ be \SI01\ formulas with $\varphi_0$ and $\psi_0$ \SI00.
Assume $(\forall x,y \in P) (\varphi(x) \land \psi(y) \implies y \npreceq
x)$.

Form the binary tree $T \subseteq 2^{<\N}$ by letting $\sigma \in T$ if and
only if $\sigma \in T(P)$ and for all $x,y< |\sigma|$:
\begin{enumerate}[\quad(i)]
 \item $(\exists m<|\sigma|)\, \varphi_0(x,m) \implies \sigma(x)=1$, and
 \item $(\exists m<|\sigma|)\, \psi_0(x,m)\implies \sigma(x)=0$.
\end{enumerate}

To see that $T$ is infinite, we show that for every $k \in \N$ there exists
$\sigma \in T$ with $|\sigma|=k$. Given $k$ let
\[ \sigma(x)=1 \iff x \in P \land (\exists y,m<k) (\varphi_0(y,m) \land x
\preceq y)\]
for all $x<k$. It is easy to verify that $\sigma\in T$. By weak K\"{o}nig's
lemma, $T$ has a path $f$. By \SI00\ comprehension, let $I=\{x \colon
f(x)=1\}$. It is straightforward to see that $I$ is as desired.

(3) is the special case of (2) obtained by considering the \SI00, and hence
\SI01, formulas $x \in A$ and $x \in B$.

It remains to prove (3) $\implies$ (1). It suffices to derive in \RCA\ from
(3) the existence of a set separating the disjoint ranges of two one-to-one
functions (\cite[Lemma IV.4.4]{sosoa}). Let $f,g \colon \N \to \N$ be
one-to-one functions such that $(\forall n,m \in \N) f(n) \neq g(m)$. Define
a partial order on $P = \{a_n, b_n, c_n\colon n \in \N\}$ by letting $c_n
\preceq a_m$ if and only if $f(m)=n$, $b_m \preceq c_n$ if and only if
$g(m)=n$, and adding no other comparabilities. Let $A = \{a_n\colon n \in
\N\}$ and $B = \{b_n\colon n \in \N\}$, so that $(\forall x \in A) (\forall y
\in B) y \npreceq x$. By (3) there exists an initial interval $I$ of $P$ such
that $A \subseteq I$ and $B \cap I = \emptyset$. It is easy to check that
$\{n\colon c_n \in I\}$ separates the range of $f$ from the range of $g$.
\end{proof}

\subsection{Essential unions of sets}
Our second result deals with finite unions of sets and will be applied to
finite unions of ideals.

\begin{definition}[\RCA]
Let $I \subseteq \N$. A family of sets $\{A_i \colon i \in I\}$ is
\emph{essential} if
\[ (\forall i \in I) (A_i \nsubseteq \bigcup_{j\in I, j \neq i} A_j).\]
The union of such a family is called an \emph{essential union}.
\end{definition}

Not every family of sets can be made essential without loosing elements from
the union. The simplest example is a sequence $\{A_n \colon n \in \N\}$ of
sets such that $A_n \subset A_{n+1}$ for every $n$. However the following
shows that, provably in \RCA, every finite family of sets can be made
essential.

\begin{lemma}[\RCA]\label{essential}
For every family of sets $\{A_i \colon i \in F\}$ with $F$ finite there
exists $I \subseteq F$ such that $\{A_i \colon i\in I\}$ is essential and
\[\bigcup _{i \in F} A_i = \bigcup_{i \in I} A_i.\]
\end{lemma}
\begin{proof}
Let
\[
n_0 = \min \{n \colon (\exists I \subseteq F) (|I|=n \land \bigcup _{i
\in F} A_i = \bigcup_{i \in I} A_i)\}.
\]
\RCA\ proves that $n_0$ exists, otherwise by \SI01-induction one could prove
\[
(\forall n) (\forall I \subseteq F) (|I| \leq n \rightarrow \bigcup _{i \in
F} A_i \neq \bigcup_{i \in I} A_i),
\]
which is clearly false.

If $I \subseteq F$ is such that $|I|=n_0$ and $\bigcup _{i \in F} A_i =
\bigcup_{i \in I} A_i$ then it is immediate that $\{A_i \colon i\in I\}$ is
essential.
\end{proof}

\section{The left to right directions}\label{section:->}

In this section we study Theorem \ref{ET1} and the left to right directions
of Theorems \ref{Bonnet1}, \ref{ET2}, and \ref{Bonnet3}. It turns out that
the left to right direction of Theorem \ref{Bonnet3} is equivalent to \ATR\
and the other statements are equivalent to \ACA.

\subsection{Equivalences with \ACA}\label{section:ACA}

We consider the following equivalence, which includes Theorems \ref{ET1} and
\ref{ET2}.

\begin{theorem}\label{strong antichain}
Let $P$ be a partial order. Then the following are equivalent:
\begin{enumerate}
 \item $P$ is a finite union of ideals;
 \item there is a finite bound on the size of the strong antichains in $P$;
 \item $P$ has no infinite strong antichains.
\end{enumerate}
\end{theorem}

We notice that (1) $\implies$ (2) and (2) $\implies$ (3) are easily provable
in \RCA. We show that (2) $\implies$ (1) and (3) $\implies$ (2) are provable
in \ACA.

We start with implication (2) $\implies$ (1).

\begin{lemma}[\ACA]\label{lemma 2}
Let $P$ be a partial order with no arbitrarily large finite strong
antichains. Then $P$ is a finite union of ideals.
\end{lemma}
\begin{proof}
Let $\ell \in \N$ be the maximum size of a strong antichain in $P$ and let
$S$ be a strong antichain of size $\ell$. For every $z \in S$ define by
arithmetical comprehension
\[A_z = \{x \in P \colon \text{$x$ and $z$ are compatible}\}.\]

Since $S$ is maximal with respect to inclusion it is immediate that $P =
\bigcup_{z \in S} A_z$ and it suffices to show that each $A_z$ is an ideal.

Fix $z \in S$ and $x,y \in A_z$. Let $x_0,y_0$ be such that $x \preceq x_0$,
$y \preceq y_0$, and $z \preceq x_0,y_0$. It suffices to show that $x_0$ and
$y_0$ are compatible in $A_z$. If this is not the case, $x_0$ and $y_0$ are
incompatible also in $P$ (because $\up{x_0} \subseteq \up z \subseteq A_z$).
Moreover for each $w \in S \setminus \{z\}$ each of $x_0$ and $y_0$ is
incompatible with $w$ in $P$ because $z$ and $w$ are incompatible in $P$.
Thus $(S \setminus \{z\}) \cup \{x_0,y_0\}$ is a strong antichain of size
$\ell+1$, a contradiction.
\end{proof}

To obtain (3) $\implies$ (2) of Theorem \ref{strong antichain} we are going
to use the existence of maximal (with respect to inclusion) strong
antichains. We first show that this statement is equivalent to \ACA.

\begin{lemma}\label{strong antichain existence}
Over \RCA, the following are equivalent:
\begin{enumerate}
 \item \ACA;
 \item every strong antichain in a partial order extends to a maximal strong
antichain;
 \item every partial order contains a maximal strong antichain.
\end{enumerate}
\end{lemma}
\begin{proof}
We show (1) $\implies$ (2). Let $P$ be a partial order and $S \subseteq P$ be
a strong antichain. By recursion we deﬁne $f\colon \N \to \{0,1\}$ by letting
$f(x) = 1$ if and only if $S \cup \{y<x \colon f(y)=1\} \cup \{x\}$ is a
strong antichain in $P$. Then $T = \{x \colon f(x) = 1\}$ is a maximal strong
antichain with $S \subseteq T$.

Implication (2) $\implies$ (3) is trivial. To show (3) $\implies$ (1), we
argue in \RCA\ and derive from (3) the existence of the range of any
one-to-one function. Given $f \colon \N \to \N$ one-to-one consider $P=\{a_n,
b_n, c_n \colon n \in \N\}$. For all $n,m \in \N$ let $a_n \preceq c_m$ if
and only if $b_n \preceq c_m$ if and only if $f(m)=n$, and add no other
comparabilities. By (3), let $S \subseteq P$ be a maximal strong antichain.
Then, $n$ belongs to the range of $f$ if and only if $a_n \notin S \lor b_n
\notin S$. Thus the range of $f$ exists by \SI00 comprehension.
\end{proof}

The following is implication (3) $\implies$ (2) of Theorem \ref{strong
antichain}, i.e.\ our formalization of the left to right direction of Theorem
\ref{ET1}.

\begin{lemma}[\ACA]\label{lemma 4}
Let $P$ be a partial order with no infinite strong antichains. Then there are
no arbitrarily large finite strong antichains in $P$.
\end{lemma}
\begin{proof}
Suppose for a contradiction that $P$ has arbitrarily large finite strong
antichains but no infinite strong antichains (the existence of such a pair is
proved below). We define by recursion a sequence of elements $(x_n, y_n) \in
P^2$.

Let $(x_0, y_0)$ be a pair such that $x_0$ and $y_0$ are incompatible in $P$
and $\up{x_0}$ contains arbitrarily large finite strong antichains. Suppose
we have defined $x_n$ and $y_n$. Using arithmetical comprehension, search for
a pair $(x_{n+1}, y_{n+1})$ such that $x_n \preceq x_{n+1}, y_{n+1}$,
$x_{n+1}$ and $y_{n+1}$ are incompatible in $P$, and $\up{x_{n+1}}$ contains
arbitrarily large finite strong antichains.

To show that the recursion never stops assume that $U \subseteq P$ is a final
interval with arbitrarily large finite strong antichains ($U=P$ at stage $0$,
$U=\up{x_n}$ at stage $n+1$). By Lemma \ref{strong antichain existence} there
exists a maximal strong antichain $S \subseteq U$ with at least two elements.
By hypothesis, $S$ is finite and we apply the following claim:

\begin{claim}
There exists $x \in S$ such that $\up x$ contains arbitrarily large finite
strong antichains.
\end{claim}
\begin{proof}[Proof of claim]
Let $n=|S|$. We first show that for every $k \geq 1$ there exists $u \in S$
such that $\up u$ contains a strong antichain of size $k$.

Given $k \geq 1$, let $T$ be a strong antichain of size $n \cdot k$. Since
$S$ is maximal, every element $y \in T$ is compatible with some element of
$S$. For any $y \in T$ let $(u(y), v(y))$ be the least pair such that $u(y)
\in S$ and $u(y), y \preceq v(y)$. Then $\{v(y) \colon y \in T\}$ is again a
strong antichain of size $n \cdot k$. As $y \mapsto u(y)$ defines a function
from $T$ to $S$, it easily follows that for some $u \in S$ the upper cone
$\up u$ contains at least $k$ elements of the form $v(y)$ with $y \in T$.

Now, for all $k \geq 1$, let $u_k \in S$ be such that $\up{u_k}$ contains a
strong antichain of size $k$. Since $S$ is finite, by the infinite pigeonhole
principle (which is provable in \ACA), there exists $x \in S$ such that
$x=u_k$ for infinitely many $k$. The upper cone $\up x$ thus contains
arbitrarily large finite strong antichains.
\end{proof}

In particular, $x_n \preceq y_m$ for all $n<m$ and $x_n$ and $y_n$ are
incompatible in $P$. It follows that $y_n$ is incompatible with $y_m$ for all
$n<m$. Then $\{y_n \colon n \in \N\}$ is an infinite strong antichain, for
the desired contradiction.
\end{proof}

The following Theorem shows that our use of \ACA\ in several of the preceding
Lemmas is necessary and establish the reverse mathematics results about
Theorem \ref{ET1} and the left to right directions of Theorems \ref{Bonnet1}
and \ref{ET2} (these are respectively conditions (3), (5), and (4) in the
statement of the Theorem). We also show that apparently weaker statements,
such as the restriction of Theorems \ref{Bonnet1} and \ref{ET2} to
well-partial orders, require \ACA.

\begin{theorem}\label{equivACA}
Over \RCA, the following are pairwise equivalent:
\begin{enumerate}
 \item \ACA;
 \item every partial order with no arbitrarily large finite strong
     antichains is a finite union of ideals;
 \item every partial order with no infinite strong antichains does not
     contain arbitrarily large finite strong antichains;
 \item every partial order with no infinite strong antichains is a finite
     union of ideals;
\item if a partial order has no infinite antichains then every initial
    interval is a finite union of ideals;
 \item every well-partial order is a finite union of ideals.
\end{enumerate}
\end{theorem}
\begin{proof}
(1) $\implies$ (2) is Lemma \ref{lemma 2} and (1) $\implies$ (3) is Lemma
\ref{lemma 4}. The combination of Lemma \ref{lemma 4} and Lemma \ref{lemma 2}
shows (1) $\implies$ (4). Since a strong antichain in a subset of a partial
order is an antichain, (4) $\implies$ (5) holds. For (5) $\implies$ (6),
recall that, provably in \RCA, a well-partial order has no infinite
antichains.

It remains to show that each of (2), (3) and (6) implies \ACA. Reasoning in
\RCA\ fix a one-to-one function $f \colon \N \to \N$. In each case we build a
suitable partial order $P$ which encodes the range of $f$.\smallskip

We start with (2) $\implies$ (1). Let $P = \{a_n, b_n \colon n \in \N\} \cup
\{c\}$. We define a partial order on $P$ by letting:
\begin{enumerate}[\quad(i)]
\item $a_n \preceq c$ for all $n$;
\item $b_n \preceq b_m$ for $n \leq m$;
\item $a_n \preceq b_m$ if and only if $(\exists i<m) f(i)=n$;
\end{enumerate}
and adding no other comparabilities. It is easy to verify that every strong
antichain in $P$ has at most $2$ elements. By (2) $P$ is a finite union of
ideals $A_0, \dots, A_k$. By Lemma \ref{essential}, we may assume that this
union is essential. Let us assume $b_0 \in A_0$.

By \SI01-induction (actually \SI00) we prove that $(\forall m)(b_m\in A_0)$.
The base step is obviously true. Suppose $b_m\in A_0$ and $b_{m+1}\notin
A_0$. Then $A_0=\{x\in P\colon x\preceq b_m\}$ (because every element $\succ
b_m$ is $\succeq b_{m+1}$). Suppose $b_{m+1}\in A_1$. Then $A_0 \subseteq
A_1$ and the decomposition is not essential, a contradiction. Therefore,
$A_0$ contains all the $b_m$'s. Now, it is straightforward to see that
$(\exists m) f(m)=n$ if and only if $a_n \in A_0$, so that the range of $f$
can be defined by \DE00 comprehension.\smallskip

To prove (3) $\implies$ (1) we exploit the notion of false and true stage.
Recall that $n \in \N$ is said to be a \emph{false stage for} $f$ (or simply
\emph{false}) if $f(k)<f(n)$ for some $k>n$ and \emph{true} otherwise. We may
assume to have infinitely many false stages, since otherwise the range of $f$
exists by \DE01 comprehension. On the other hand, there are always infinitely
many true stages (i.e.\ for every $m$ there exists $n>m$ which is true),
because otherwise we can build an infinite descending sequence of natural
numbers.

Let $P=\{a_n, b_n \colon n \in \N\}$ and define
\begin{enumerate}[\quad(i)]
 \item $b_n \preceq b_m$ for all $n<m$;
 \item $a_n \preceq b_m$ if and only if $f(k)<f(n)$ for some $k$ with $n<k
     \leq m$ (i.e.\ if at stage $m$ we know that $n$ is false);
\end{enumerate}
and there are no other comparabilities.

Notice that the $b_n$'s and the $a_n$'s with $n$ false are pairwise
compatible in $P$. Therefore every infinite strong antichain in $P$ consists
of infinitely many $a_n$'s with $n$ true and at most one $b_n$ or $a_n$ with
$n$ false. Possibly removing that single element we have an infinite set of
true stages. From this in \RCA\ we can obtain a strictly increasing
enumeration of true stages $i \mapsto n_i$. Since $(\exists n) f(n)=m$ if and
only if $(\exists n \leq n_m) f(n)=m$, the range of $f$ exists by \DE01
comprehension. Thus the existence of an infinite strong antichain in $P$
implies the existence of the range of $f$ in \RCA.

To apply (3) and conclude the proof we need to show that $P$ contains
arbitrarily large finite strong antichains. To do this apparently we need
\SI02-induction (which is not available in \RCA) to show that for all $k$
there exists $k$ distinct true stages.

To remedy this problem (with the same trick used for this purpose in
\cite[Lemma 4.2]{MarSho11}) we replace each $a_n$ with $n+1$ distinct
elements. Thus we set $P'=\{a_n^i, b_n \colon n \in \N, i \leq n\}$ and
substitute (ii) with $a_n^i \leq_{P'} b_m$ if and only if $f(k)<f(n)$ for
some $k$ with $n<k \leq m$. Then also the existence of an infinite strong
antichain in $P'$ suffices to define the range of $f$ in \RCA. However the
existence of arbitrarily large finite strong antichains in $P'$ of the form
$\{a_n^i \colon i \leq n\}$ follows immediately from the existence of
infinitely many true stages.\smallskip

We now show (6) $\implies$ (1). We again use false and true stages and as
before we assume to have infinitely many false stages. The idea for $P$ is to
combine a linear order $P_0 = \{a_n \colon n \in\ N\}$ of order type $\omega
+ \omega^*$ with a linear order $P_1 = \{b_n \colon n \in \N\}$ of order type
$\omega$. The false and true stages give rise respectively to the $\omega$
and $\omega^*$ part of $P_0$, and every false stage is below some element of
$P_1$. We proceed as follows.

Let $P = \{a_n, b_n \colon n\in\N\}$. For $n \leq m$, set
\begin{enumerate}[\quad(i)]
 \item $a_n\preceq a_m$ if $f(k)<f(n)$ for some $n<k\leq m$ (i.e.\ if at
     stage $m$ we know that $n$ is false);
 \item $a_m\preceq a_n$ if $f(k)>f(n)$ for all $n<k\leq m$ (i.e.\ if at
     stage $m$ we believe $n$ to be true).
\end{enumerate}
When condition (i) holds, we also put $a_n \preceq b_m$. Then we linearly
order the $b_m$'s by putting $b_i \preceq b_j$ if and only if $i \leq j$.
There are no other comparabilities.

It is not difficult to verify that $P$ is a partial order with no infinite
antichains. Note that if $n$ is false and $m>n$ is such that $f(m)<f(n)$,
then $\{i \colon a_i \preceq a_n\} \subseteq \{i \colon i<m\}$ is finite,
while if $n$ is true, then $\{i \colon a_n \preceq a_i\} \subseteq \{i \colon
i \leq n\}$ is finite. This explains our assertion that $P_0$ has order type
$\omega + \omega^*$.

First assume that $P$ is not a well-partial order. By definition, there
exists $g \colon \N \to P$ such that $i<j$ implies $g(i) \npreceq g(j)$. As
for every false $n$ there are only finitely many $x \in P$ such that $a_n
\npreceq x$, we must have $g(i) \neq a_n$ for all $i$ and for all false $n$.
We may assume that $g(i) \neq b_n$ for all $i,n$, since there are finitely
many $b_m$ such that $b_n \npreceq b_m$. We thus have $g(i) = a_{n_i}$ with
$n_i$ true for all $i$. Since $a_m \succ a_n$ and $n<m$ imply $n$ false, the
map $i \mapsto n_i$ is a strictly increasing enumeration of true stages. As
before, the range of $f$ exists by \DE01 comprehension.

We now assume that $P$ is a well-partial order. Apply (6), so that
$P=\bigcup\{A_i \colon i<k\}$ is a finite union of ideals. By Lemma
\ref{essential} we may assume that the union is essential so that there
exists an ideal, say $A_0$, that contains all the $b_m$'s.

We claim that $n$ is false if and only if $a_n\in A_0$. To see this, let $n$
be false. Thus $a_n\preceq b_m$ for some $m$, and hence $a_n\in A_0$.
Conversely, if $a_n\in A_0$ then it is compatible with, for instance, $b_0$,
and yet again it is $\preceq b_m$ for some $m$. Hence, the set of true stages
is $\{n\colon a_n\notin A_0\}$, and the conclusion follows as before.
\end{proof}

\subsection{Equivalences with \ATR}\label{section:uncountably}

We now consider the left to right direction of Theorem \ref{Bonnet3}, i.e.\
the statement every countable scattered partial order with no infinite
antichains has countably many initial intervals. We start with a technical
Lemma:

\begin{lemma}[\ACA]\label{pm to um}
If a partial order $P$ has perfectly many initial intervals, then there
exists $x\in P$ such that either
\begin{enumerate}[\quad (i)]
 \item $\inc x$ has uncountably many initial intervals, or
 \item both $\downs x$ and $\ups x$ have uncountably many initial
     intervals.
\end{enumerate}
\end{lemma}
\begin{proof}
Let $P$ be a partial order with perfectly many initial intervals. Let $T
\subseteq T(P)$ be a perfect tree.

We first show that there exist $x \in P$ such that both
\[ \{I \in \Int(P) \colon x \notin I\} \text{ and } \{I \in \Int(P) \colon x
\in
I\}\]
are uncountable. Let $\tau \in T$ be such that both $\tau_0 = \tau \conc
\langle 0 \rangle$ and $\tau_1 = \tau \conc \langle 1 \rangle$ belong to $T$.
Let $x = |\tau|$ and notice that $x \in P$. For $i<2$ define $T_i = \{\sigma
\in T \colon \sigma \sqsubseteq \tau_i \lor \tau_i \sqsubseteq \sigma\}$. The
trees $T_0$ and $T_1$ are perfect and witness the fact that the two
collections of initial intervals are uncountable.\smallskip

Now, suppose that condition (i) fails and let $\Int (\inc x) = \{J_n \colon n
\in \N\}$. We aim to show that (ii) holds.

Suppose for a contradiction that $\downs x$ has countably many initial
intervals and let $\Int (\downs x) = \{I_n \colon n \in \N\}$. Then it is not
difficult to show that
\[ \{I \in \Int(P) \colon x \notin I\} = \{ I_n \cup \Down J_m \colon n,m \in
\N\}. \]
This contradicts the fact that $\{I \in \Int(P) \colon x \notin I\}$ is
uncountable.

Similarly, suppose that $\ups x$ has countably many initial intervals and let
$\Int (\ups x) = \{I_n \colon n \in \N\}$. Then, it is not difficult to show
that
\[
\{I \in \Int(P) \colon x \in I\} =
\{ \Down (\{x\} \cup I_n \cup J_m) \colon n,m \in \N\}.
\]
This contradicts the fact that $\{I \in \Int(P) \colon x \in I\}$ is
uncountable. Therefore, condition (ii) holds.
\end{proof}

\begin{theorem}[\ATR]\label{Bonnet3->}
Every scattered partial order with no infinite antichains has countably many
initial intervals.
\end{theorem}
\begin{proof}
Let $P$ be a partial order with uncountably many initial intervals.

Let $\Fin(P)$ the set of (codes for) finite subsets of $P$. For all $F, G, H
\in \Fin(P)$,  let
\[ P_{F,G,H} = \bigcap_{x \in F} \downs x \cap \bigcap_{x \in G} \ups x \cap
   \bigcap_{x \in H} \inc x. \]

We want to define a pruned tree $T \subseteq 3^{<\N}$ and a function $f
\colon T \to \Fin(P)^3$ such that the following hold (where $f(\sigma) =
(F_\sigma, G_\sigma, H_\sigma)$ and $P_\sigma = P_{f(\sigma)}$):
\begin{enumerate}[\quad(i)]
 \item $f(\langle \rangle) = (\emptyset, \emptyset, \emptyset)$;
 \item for all $\sigma \in T$, $\si0 \in T$ if and only if $\si1 \in T$ if
     and only if $\si2 \notin T$ (in other words there are two
     possibilities: either exactly \si0 and \si1 belong to $T$, or only
     $\si2 \in T$);
 \item if $\si0 \in T$, then $f(\si0) = (F_{\sigma} \cup \{x\}, G_\sigma,
     H_\sigma)$ and $f(\si1) = (F_\sigma, G_\sigma \cup \{x\}, H_\sigma)$
     for some $x \in P_\sigma$;
 \item if $\si2 \in T$, then $f(\si2) = (F_\sigma, G_\sigma, H_\sigma \cup
     \{x\})$ for some $x \in P_\sigma$.
\end{enumerate}

We first show that if there exist $T$ and $f$ as above, then $P$ is not
scattered or it contains an infinite antichain.

First suppose there exists a path $g \in [T]$ such that $g(n)=2$ for
infinitely many $n$. Then let
\[ D = \bigcup_{n \in \N} H_{g \restr n}. \]
It is easy to check, using (iv) and the definition of $P_{F,G,H}$, that $D$
is an infinite antichain.

If there are no paths $g \in [T]$ such that $g(n)=2$ for infinitely many $n$
then it is easy to see, using (ii), that $T$ is perfect. For all $\si0 \in
T$, let $x_\sigma$ be the unique element of $F_{\si0} \setminus F_\sigma$. We
claim that
\[ Q = \{ x_\sigma \colon \si0 \in T\} \]
is a dense linear order in $P$.

We first note that $x_\sigma \neq x_\tau$ for $\sigma, \tau \in T$ with
$\sigma \neq \tau$. Now fix distinct $x_\sigma, x_\tau \in Q$ with the goal
of showing that they are comparable in $P$ and that there exists an element
of $Q$ strictly between them. First assume that $\sigma$ and $\tau$ are
comparable as sequences, let us say $\sigma \sqsubset \tau$. Then, using
(iii), $x_\tau \prec x_\sigma$ if $\si0 \sqsubseteq \tau$ and $x_\sigma \prec
x_\tau$ if $\si1 \sqsubseteq \tau$. Suppose $x_\tau \prec x_\sigma$ (the
other case is similar) and let $\eta \in T$ so that $\tau \conc \langle 1
\rangle \sqsubseteq \eta$ and $\eta \conc \langle 0 \rangle \in T$. Then
$x_\tau \prec x_\eta \prec x_\sigma$ by (iii). Suppose now that $\sigma$ and
$\tau$ are not one initial segment of the other. We may assume that $\eta
\conc \langle 0 \rangle \sqsubseteq \sigma$ and $\eta \conc \langle 1 \rangle
\sqsubseteq \tau$ for some $\eta$. Then $x_\eta \in Q$ and, using (iii)
again, $x_\sigma \prec x_\eta \prec x_\tau$.\smallskip

It remains to show that we can define $T$ and $f$ satisfying (i)--(iv).

By Theorem \ref{countable or perfect}, $P$ has perfectly many initial
intervals. Let $U$ be a perfect subtree of $T(P)$. By Corollary \ref{model
existence}, there exists a countable coded $\omega$-model $M$ such that $P, U
\in M$ and $M$ satisfies \DC\ and \ATRl{P}.

We recursively define $T$ and $f$ using $M$ as a parameter. Let $\langle
\rangle \in T$ and $f(\langle \rangle) = (\emptyset, \emptyset, \emptyset)$
as required by (i). Note that $M$ satisfies ``$T(P_{\langle \rangle})$
contains a perfect subtree''. Let $\sigma \in T$ and assume by arithmetical
induction that $M$ satisfies ``$T(P_\sigma)$ contains a perfect subtree''.
Since $M$ is a model of \ACA, by Lemma \ref{pm to um} applied to $P_\sigma$,
there exists $x \in P_\sigma$ such that either
\begin{enumerate}[\quad(a)]
 \item $M$ satisfies ``$T(P_\sigma \cap \inc x)$ has uncountably many
     paths'', or
 \item $M$ satisfies ``both $T(P_\sigma \cap \downs x)$ and $T(P_\sigma
     \cap \ups x)$ have uncountably many paths''.
\end{enumerate}
Search for the least $x$ with this arithmetical property. If (a) holds (and
we can check this arithmetically outside $M$), use \ATRl{P} within $M$ to
apply Theorem \ref{PTT lightface} to the $P$-computable tree $T(P_\sigma \cap
\inc x)$. We obtain that $M$ satisfies ``$T(P_\sigma \cap \inc x)$ contains a
perfect subtree''. Thus, let $\si2 \in T$ and set $f(\si2) = (F_\sigma,
G_\sigma, H_\sigma \cup \{x\})$. If (b) holds, then arguing analogously we
obtain that $M$ satisfies ``both $T(P_\sigma \cap \downs x)$ and $T(P_\sigma
\cap \ups x)$ contain perfect subtrees''. Thus let $\si0, \si1 \in T$ and set
\[f(\si0) = (F_\sigma \cup \{x\}, G_\sigma, H_\sigma) \text{ and }
  f(\si1) = (F_\sigma, G_\sigma \cup \{x\}, H_\sigma).\]
In any case, (ii)-(iv) are satisfied and the induction hypothesis that $M$
satisfies ``$T(P_\sigma)$ contains a perfect subtree'' is preserved.
\end{proof}

\begin{theorem}\label{Bonnet3->ATR}
Over \ACA, the following are equivalent:
\begin{enumerate}
 \item \ATR;
 \item every scattered partial order with no infinite antichains has
     countably many initial intervals;
 \item every scattered linear order has countably many initial intervals.
\end{enumerate}
\end{theorem}
\begin{proof}
(1) $\implies$ (2) is Theorem \ref{Bonnet3->} and (2) $\implies$ (3) is
immediate. We show (3) $\implies$ (1) by essentially repeating the proof of
\cite[Theorem 18]{Clo89}.

Assume \ACA. We wish to prove \ATR. By \cite[Theorem V.5.2]{sosoa}, \ATR\ is
equivalent (over \RCA) to the statement asserting that for every sequence of
trees $\{T_i \colon i \in \N\}$ such that every $T_i$ has at most one path,
there exists the set $\{i \in \N \colon [T_i] \neq \emptyset\}$. So let
$\{T_i \colon i \in \N\}$ be such a sequence. Let us order each $T_i$ with
the Kleene-Brouwer order $\leq_\KB$ and define the linear order $L = \sum_{i
\in \N} T_i$

We aim to show that $L$ is scattered. By Lemma \ref{scattls}, it suffices to
prove that every $T_i$ is scattered. To this end, we show that if a tree $T$
has at most one path then the Kleene-Brouwer order on $T$ is of the form
\[    \tag{$*$} X + \sum_{n \in \omega^*} Y_n, \]
where $X$ and the $Y_n$ are (possibly empty) well-orders. Applying Lemma
\ref{scattls} again, we obtain that $T$ is scattered.

If $T$ has no path, then \ACA\ proves that $\leq_\KB$ well-orders $T$, and
hence we can take $X=T$ and the $Y_n$'s empty. Now let $f$ be the unique path
of $T$. Let $X = \{\sigma \in T \colon (\forall n) \sigma<_\KB f \restr n\}$
and $Y_n = \{\sigma \in T \colon f \restr n+1<_{\KB} \sigma \leq_\KB f \restr
n\}$, for all $n \in \N$. It is straightforward to see that $(*)$ holds. We
now claim that $X$ is a well-order. Suppose not, and let $(\sigma_n)_{n \in
\N}$ be an infinite descending sequence in $X$. Form the tree $T_0 = \{\sigma
\in T \colon (\exists n) \sigma \sqsubseteq \sigma_n\}$. Then $T_0$ is not
well-founded and so it has a path. As $T_0$ is a subtree of $T$, this path
must be $f$. Let $i \in \N$ be such that $\sigma_0 \restr i = f\restr i$ and
$\sigma_0(i) < f(i)$ (such an $i$ exists because $\sigma_0 \in X$). On the
other hand, $f \restr i+1 \in T_0$, and thus $f \restr i+1 \sqsubseteq
\sigma_n$ for some $n \in \N$. It follows that $\sigma_0<_\KB \sigma_n$, a
contradiction. To show that each $Y_n$ is a well-order notice that $Y_n =
\{\sigma \in T \colon f \restr n \sqsubset \sigma \land f(n) < \sigma(n)\}
\cup \{f \restr n\}$.

Apply (3) to $L$ and let $\Int(L) = \{I_n \colon n \in \N\}$. It is easy to
check that $T_i$ has a path if and only if
\[
(\exists n) \big( \bigcup_{j<i} T_j \subseteq I_n \land T_i \nsubseteq I_n \land
L \setminus I_n \text{ has no least element}\big).
\]
Therefore, the set $\{i \in \N \colon [T_i] \neq \emptyset\}$ exists by
arithmetical comprehension.
\end{proof}

It is worth noticing that a natural weakening of condition (3) of Theorem
\ref{Bonnet3->ATR} is provable in \RCA:

\begin{lemma}[\RCA]\label{Bonnet3w->RCA}
Every linear order with perfectly many initial intervals is not scattered.
\end{lemma}
\begin{proof}
Let $L$ be a linear order and $T \subseteq T(L)$ be a perfect tree. Define
\[
Q = \{x \in L \colon (\exists \sigma \in T) (|\sigma| = x \land \si0, \si1
\in T)\}.
\]
The argument showing that $Q$ is a dense subset of $L$ is similar to the one
in the proof of Theorem \ref{Bonnet3->}.
\end{proof}

\section{The right to left directions}\label{section:<-}

In this section we study the right to left directions of Theorems
\ref{Bonnet1}, \ref{ET2}, and \ref{Bonnet3}. The right to left direction of
Theorem \ref{Bonnet3} naturally splits into two statements with the same
hypothesis (the existence of countably many initial intervals) and different
conclusions (the partial order is scattered and the partial order has no
infinite antichains). We have thus four different statements altogether. All
these statements have simple proofs in \ACA, but it turns out that each of
them can be proved in a properly weaker system.

\subsection{Proofs in \RCA}\label{section:RCA}

We start with a simple observation about the right to left direction of
Theorem \ref{ET2}.

\begin{lemma}[\RCA]\label{ET2<-}
Every partial order which is a finite union of ideals has no infinite strong
antichains.
\end{lemma}
\begin{proof}
Since an ideal does not contain incompatible elements, if the partial order
is the union of $k$ ideals we have even a finite bound on the size of strong
antichains.
\end{proof}

Another statement that can be proved in \RCA\ is the following half of the
right to left direction of Theorem \ref{Bonnet3}.

\begin{theorem}[\RCA]\label{Bonnet3<-RCA}
Every partial order with countably many initial intervals is scattered.
\end{theorem}
\begin{proof}
We show that if $P$ is not scattered, then $P$ has perfectly many initial
intervals. By Lemma \ref{scatt} we may assume that $P$ contains a dense
linear order $Q$.

We define by recursion an embedding $f \colon 2^{<\N} \to T(P)$. Thus $T_0 =
\{\tau \in T(P) \colon (\exists \sigma \in 2^{<\N}) \tau \sqsubseteq
f(\sigma)\}$ is a perfect subtree of $T(P)$. Since $\tau \in T_0$ if and only
if $(\exists \sigma \in 2^{<\N})(|\sigma| = |\tau| \land \tau \sqsubseteq
f(\sigma))$, $T_0$ exists in \RCA.

We say that $x \in P$ is \emph{free for} $\tau \in T(P)$ if
\[ (\forall y < |\tau|) ((\tau(y) =1 \implies x \npreceq y) \land (\tau(y)=0
\implies y \npreceq x)). \]
In other words, $x$ is free for $\tau$ if and only if there exist $\tau_0,
\tau_1 \in T(P)$ with $\tau \sqsubset \tau_i$ and $\tau_i(x)=i$. Since $T(P)$
is a pruned tree this means that there exist two initial intervals of $P$
whose characteristic function extends $\tau$, one containing $x$ and the
other avoiding $x$.

Let $f(\langle \rangle) = \langle \rangle$. Suppose we have defined
$f(\sigma) = \tau$. Assume by \SI01 induction that $Q$ contains at least two
(and hence infinitely many) elements that are free for $\tau$. Then search
for  $a \prec b \prec c$ in $Q$ that are free for $\tau$.  We will define
$\tau_0, \tau_1 \in T(P)$ which are extensions of $\tau$ with $|\tau_i| =
b+1$ and $\tau_i(b) =i$. Thus $\tau_0$ and $\tau_1$ are incompatible and we
can let $f(\si i)=\tau_i$.

We show how to define $\tau_0$ (to define $\tau_1$ replace $a$ with $b$ and
$b$ with $c$). Since $\{x \in P \colon x<b\}$ is finite, we can find $a', b'
\in Q$ with $a \prec a' \prec b' \prec b$ such that $a',b'>b$, and for no $x
\in P$ with $x<b$ we have $a' \prec x \prec b'$. Given $x < |\tau_0|$ we need
to define $\tau_0(x)$, and we proceed by cases (notice that the first three
conditions are determined by the fact that we want $\tau_0 \in T(P)$ and
$\tau_0 \sqsupseteq \tau$):
\begin{itemize}
  \item if $x \notin P$ let $\tau_0 (x) =0$;
  \item if $x \in P$ is not free for $\tau$ because there exists $y <
      |\tau|$ such that $\tau(y) =0$ and $y \preceq x$ let $\tau_0 (x) =
      0$;
  \item if $x \in P$ is not free for $\tau$ because there exists $y <
      |\tau|$ such that $\tau(y) =1$ and $x \preceq y$ let $\tau_0 (x) =
      1$;
  \item if $x$ is free for $\tau$ we define $\tau_0(x)$ according to the
      following cases:
\begin{enumerate}[\quad(i)]
 \item if $x \prec a'$, let $\tau_0(x) =1$;
 \item if $x \succ b'$, let $\tau_0(x) =0$;
 \item otherwise, let $\tau_0(x) =0$.
\end{enumerate}
\end{itemize}
It is not difficult to check that $\tau_0$ extends $\tau$, $\tau_0(b) =0$ and
both $a'$ and $b'$ are free for $\tau_0$, preserving the induction
hypothesis.
\end{proof}

With regard to the other half of the right to left direction of Theorem
\ref{Bonnet3}, \RCA\ proves the following.

\begin{lemma}[\RCA]\label{ACpm}
An infinite antichain has perfectly many initial intervals.
\end{lemma}
\begin{proof}
If $P$ is an antichain then the tree $T(P)$ consists of all $\sigma \in
2^{<\N}$ such that $x \notin P$ implies $\sigma(x)=0$. If $P$ is infinite it
is immediate that this tree is perfect and thus Lemma \ref{T(P)} implies that
$P$ has perfectly many initial intervals.
\end{proof}

\subsection{Proofs in \WKL}\label{section:WKL}

We now look at the right to left direction of Theorem \ref{Bonnet1}, which
states that every partial order with an infinite antichain contains an
initial interval that cannot be written as a finite union of ideals. The
proof can be carried out very easily in \ACA: just take the downward closure
of the given antichain. We improve this upper bound by showing that \WKL\
suffices. We first point out that \RCA\ proves a particular instance of the
statement.

\begin{lemma}[\RCA] \label{Bonnet1<-max}
Let $P$ be a partial order with a maximal (with respect to inclusion)
infinite antichain. Then there exists an initial interval that is not a
finite union of ideals.
\end{lemma}
\begin{proof}
Let $D$ be a maximal infinite antichain of $P$. The maximality of $D$ implies
that for all $x\in P$ we have
\[ (\exists d \in D) x \preceq d \iff \neg( \exists d \in D) d\prec x.\]
Therefore the downward closure of $D$ is \DE01 definable and thus exists in
\RCA. Letting $I = \{x \in P \colon (\exists d\in D) x \preceq d\}$ we obtain
an initial interval which is not a finite union of ideals, since distinct
elements of $D$ are incompatible in $I$.
\end{proof}

To use Lemma \ref{Bonnet1<-max} in the general case we need to extend an
infinite antichain to a maximal one. It is easy to show that \RCA\ proves the
existence of maximal antichains in any partial order. On the other hand, we
now show in \RCA\ that the statement that every antichain is contained in a
maximal antichain is equivalent to \ACA.

\begin{lemma}
Over \RCA, the following are equivalent:
\begin{enumerate}
 \item \ACA;
 \item every antichain in a partial order extends to a maximal
antichain.
\end{enumerate}
\end{lemma}
\begin{proof}
We first show (1) $\implies$ (2). Let $P$ be a partial order and $D \subseteq
P$ be an antichain. By recursion we define $f \colon \N \to \{0,1\}$ by
letting $f(x)=1$ if and only if $D \cup \{y<x \colon f(y)=1\} \cup \{x\}$ is
an antichain in $P$. Then $E = \{x \colon f(x)=1\}$ is a maximal antichain
with $D \subseteq E$.

For the reversal argue in \RCA\ and fix a one-to-one function $f \colon \N
\to \N$. Let $P = \{a_n, b_n \colon n \in \N\}$ and define the partial order
by letting $b_m \preceq a_n$ if and only if $f(m) = n$, and adding no other
comparabilities. Then apply (2) to the antichain $D= \{b_m \colon m \in \N\}$
and obtain a maximal antichain $E$ such that $D \subseteq E$. It is immediate
that $(\exists m) f(m)=n$ if and only if $a_n \notin E$, so that in \RCA\ we
can prove the existence of the range of $f$.
\end{proof}

We now show how to prove the right to left direction of Theorem \ref{Bonnet1}
in \WKL.

\begin{theorem}[\WKL]\label{Bonnet1<-WKL}
Every partial order with an infinite antichain contains an initial interval
that is not a finite union of ideals.
\end{theorem}
\begin{proof}
Let $P$ be a partial order containing an infinite antichain $D$. Let
$\varphi(x)$ and $\psi(x)$ be the \SI01 formulas $x \in D$ and $(\exists y)
(y \in D \land y \prec x)$ respectively. It is obvious that $(\forall x,y \in
P) (\varphi(x) \land \psi(y) \implies y \npreceq x)$. By \SI01 initial
interval separation (Lemma \ref{iis}), there exists an initial interval $I
\subseteq P$ such that
\[(\forall x \in P) ((\varphi(x) \implies x \in I) \land (\psi(x) \implies x
\notin I)).\]
Therefore, $I$ contains $D$ and no element above any element of $D$. To see
that $I$ cannot be the union of finitely many ideals notice that distinct
$x,x' \in D$ cannot belong to the same ideal $A \subseteq I$, for otherwise
there would be $z \in I$ such that $x, x' \preceq z$, which implies
$\psi(z)$.
\end{proof}

We do not know whether the statement of Theorem \ref{Bonnet1<-WKL} implies
\WKL. Notice however that the proof above uses the existence of an initial
interval $I$ containing the infinite antichain $D$ and no elements above any
element of $D$. We now show that even the existence of an initial interval
$I$ containing infinitely many elements of the antichain $D$ and no elements
above any element of $D$ is equivalent to \WKL. Therefore a proof of the
right to left direction of Theorem \ref{Bonnet1} in a system weaker than
\WKL\ must avoid using such an $I$.

\begin{lemma}\label{WKLneeded}
Over \RCA, the following are equivalent:
\begin{enumerate}
 \item \WKL;
 \item if a partial order $P$ contains an infinite antichain $D$, then $P$ has
an initial interval $I$ such that $D\subseteq I$ and $(\forall x\in D)(\forall
y\in I)x\nprec y$;
\item if a partial order $P$ contains an infinite antichain $D$, then $P$ has
an initial interval $I$ such that $I\cap D$ is infinite and $(\forall x\in
D)(\forall y\in I)x\nprec y$.
\end{enumerate}
\end{lemma}
\begin{proof}
The proof of (1) $\implies$ (2) is contained in Theorem \ref{Bonnet1<-WKL}
and (2) $\implies$ (3) is obvious, so that we just need to show (3)
$\implies$ (1). Fix one-to-one functions $f,g \colon \N \to \N$ such that
$(\forall n,m \in \N) f(n) \neq g(m)$. Let $P = \{a_n, b_n \colon n \in \N\}$
the partial order defined by letting
\begin{enumerate}[\quad(i)]
\item $a_n \preceq b_m$ if $m=g(n)$;
\item $b_k \preceq a_n $ if $(\exists i<n)(i<g(n) \land f(i)=k)$, i.e.\ $k$
    enters the range of $f$ before stage $\min\{n,g(n)\}$;
\item $b_k \preceq b_m$ if $(\exists i<m) (f(i)=k \land (\forall j<i) f(j)
    \neq m)$, i.e.\ $k$ enters the range of $f$ before stage $m$ and when
    $m$ has not entered the range of $f$ yet,
\end{enumerate}
and adding no other comparabilities.

To check that $P$ is indeed a partial order we need to show that it is
transitive. The main cases are the following:
\begin{itemize}
  \item If $b_k \preceq a_n \preceq b_m$ we have $m=g(n)$ and the existence
      of $i<\min\{n,m\}$ such that $f(i)=k$. By the hypothesis on $f$ and
      $g$ we have $f(j) \neq m$ for every $j$, and in particular for every
      $j<i$, so that $b_k \preceq b_m$ follows.
  \item If $b_k \preceq b_m \preceq b_\ell$ there exist $i<m$ and $i'<\ell$
      such that $f(i)=k$, $(\forall j<i) f(j) \neq m$, $f(i')=m$, and
      $(\forall j<i') f(j) \neq \ell$. The second and third condition imply
      $i \leq i'$, so that $i<\ell$, $(\forall j<i) f(j) \neq \ell$ and we
      obtain $b_k \preceq b_\ell$.
  \item If $b_k \preceq b_m \preceq a_n$  there exist $i<m$ and $i'<n$ such
      that $f(i)=k$, $(\forall j<i) f(j) \neq m$, $i'<g(n)$, and $f(i')=m$.
      Again we obtain $i \leq i'$, so that $i< \min\{n,g(n)\}$ and we can
      conclude $b_k \preceq a_n$.
\end{itemize}

The set $D = \{a_n \colon n\in\N\}$ is an infinite antichain. Applying (3) we
obtain an initial interval $I$ of $P$ which contains infinitely many elements
of $D$ and no elements above any element of $D$. We now check that $\{k\in\N
\colon b_k\in I\}$ separates the range of $f$ from the range of $g$.

If $k=g(n)$ it is immediate that $a_n \prec b_k$ so that $b_k \notin I$.

On the other hand suppose that $k=f(i)$. The set $A= \{n \colon g(n) \leq
i\}$ is finite by the injectivity of $g$ and we can let $m = \max (\{i\} \cup
A)$. Since $D \cap I$ is infinite there exists $n>m$ such that $a_n \in I$.
Then we have $i<n$ and $i<g(n)$ (because $n \notin A$), so that $b_k \preceq
a_n$. Therefore $b_k \in I$.
\end{proof}

We notice that another weakening of statement (2) of Lemma \ref{WKLneeded}
which is equivalent to \WKL\ is the following: \lq\lq if a partial order $P$
contains an infinite antichain $D$, then there exists an initial interval $I$
such that $D \subseteq I$ and $(\forall y \in I)(\exists^\infty x \in D)x
\nprec y$\rq\rq\ (the proof of the reversal uses the partial order of the
proof above equipped with the inverse order). However this statement does not
imply the statement of Theorem \ref{Bonnet1<-WKL}.\smallskip

Our next goal is to show that \WKL\ suffices to prove the half of the right
to left direction of Theorem \ref{Bonnet3} that was not proved in \RCA\ in
Theorem \ref{Bonnet3<-RCA}. In other words, we study the statement that every
partial order with countably many initial intervals has no infinite
antichains. Understanding initial intervals of partial orders with an
infinite antichain leads to study the relationship between initial intervals
of partial orders contained one into the other.

\begin{lemma}\label{embedii}
Over \RCA, the following are equivalent:
\begin{enumerate}
\item \WKL;
\item Let $Q$ and $P$ be partial orders and $f$ be an embedding of $Q$ into
    $P$. Then
\[ \Int(Q) = \{f^{-1}(J) \colon J \in \Int(P)\};  \]
\item Let $Q$ be a subset of a partial order $P$. Then $\Int(Q) = \{J \cap
    Q \colon J \in \Int(P)\}$.
\end{enumerate}
\end{lemma}
\begin{proof}
We start with (1) $\implies$ (2). Let $f \colon Q \to P$ be an embedding. It
is easy to check that if $J \in \Int(P)$ then $f^{-1}(J) \in \Int(Q)$, so
that the right to left inclusion is established even in \RCA.

For the other inclusion fix $I \in \Int(Q)$. Let $\varphi(x)$ and $\psi(x)$
be the \SI01 formulas $(\exists y \in Q) (y \in I \land x=f(y))$ and
$(\exists y \in Q) (y \notin I \land x=f(y))$ respectively. Since $f$ is an
embedding and $I$ is an initial interval, we have
\[  (\forall x,y \in P) (\varphi(x) \land \psi(y) \implies y \npreceq_P x). \]
Apply \SI01 initial interval separation (Lemma \ref{iis}) to get $J \in
\Int(P)$ such that $f(I) \subseteq J$ and $J \cap f(Q \setminus I) =
\emptyset$. It is immediate that $I = f^{-1}(J)$.\smallskip

Since the implication (2) $\implies$ (3) is obvious, it remains to show (3)
$\implies$ (1).

Instead of \WKL, we prove statement (3) of Lemma \ref{iis}, i.e.\ initial
interval separation. Let $P$ be a partial order and $A, B \subseteq P$ such
that $(\forall x \in A) (\forall y \in B) y \npreceq x$. Let $Q = A \cup B
\subseteq P$ and notice that $A \in \Int(Q)$. By (3) there exists $J \in
\Int(P)$ such that $A = J \cap Q$. It is easy to see that $A \subseteq J$ and
$J \cap B = \emptyset$, completing the proof.
\end{proof}

Notice that the obvious proof of the nontrivial direction of (2), namely
given $I \in \Int(Q)$ let $J$ be the downward closure of $f(I)$, uses
arithmetical comprehension.

\begin{corollary}[\WKL]\label{embedding}
Let $P$ and $Q$ be partial orders such that $Q$ embeds into $P$. If $P$ has
countably many initial intervals, then $Q$ does.
\end{corollary}
\begin{proof}
Fix an embedding $f \colon Q \to P$. Let $\{J_n \colon n \in \N\}$ be such
that for all $J \in \Int(P)$ there exists $n$ with $J=J_n$. For every $n$ let
$I_n = f^{-1} (J_n)$, which exists in \RCA. Then, by Lemma \ref{embedii}, for
all $I \in \Int(Q)$ there exists $n$ with $I=I_n$, showing that $Q$ has
countably many initial intervals.
\end{proof}

We can now prove in \WKL\ the part of the right to left direction of Theorem
\ref{Bonnet3} we are interested in.

\begin{theorem}[\WKL]\label{Bonnet3<-WKL}
Every partial order with countably many initial intervals has no infinite
antichains.
\end{theorem}
\begin{proof}
Immediate from Lemma \ref{ACpm} and Corollary \ref{embedding}.
\end{proof}

\subsection{Unprovability in \RCA}\label{section:REC}

In this subsection we show that \RCA\ does not suffice to prove the
statements shown in Theorems \ref{Bonnet1<-WKL} and \ref{Bonnet3<-WKL} to be
provable in \WKL.

A single construction actually works for both statements.

\begin{lemma}\label{REC}
There exists a computable partial order $P$ with an infinite computable
antichain such that any computable initial interval of $P$ is the downward
closure of a finite subset of $P$.
\end{lemma}

Before proving Lemma \ref{REC} we show how to deduce from it the
unprovability results.

\begin{theorem}\label{Bonnet1<-nRCA}
\RCA\ does not prove that every partial order such that all its initial
intervals are finite union of ideals has no infinite antichains.
\end{theorem}
\begin{proof}
It suffices to show that the statement fails in \REC, the $\omega$-model of
computable sets. Let $P$ the computable partial order of Lemma \ref{REC} and
let $I$ be a computable initial interval of $P$. Let $F$ be a finite set such
that $I = \Down F$. Then $I = \bigcup_{x \in F} \down x$ and each $\down x$
is a computable ideal.

Thus all initial intervals of $P$ which belong to \REC\ are finite union of
ideals also belonging to \REC. On the other hand, $P$ has an infinite
antichain in \REC, showing the failure of the statement.
\end{proof}

\begin{theorem}\label{Bonnet3<-nRCA}
\RCA\ does not prove that every partial order with countably many initial
intervals has no infinite antichains.
\end{theorem}
\begin{proof}
We again show that the statement fails in \REC, and again use the computable
partial order $P$ of Lemma \ref{REC}. Since the downward closures of finite
subsets of $P$ are uniformly computable, there exists a set $\{I_n \colon n
\in N\}$ in \REC\ which lists all computable initial intervals of $P$.
Therefore \REC\ satisfies that $P$ has countably many initial intervals.
Since $P$ has an infinite antichain in \REC, the statement fails.
\end{proof}

\begin{proof}[Proof of Lemma \ref{REC}]
We build $P$ by a finite injury priority argument. We let $P = \{x_n, y_n
\colon n \in \omega\}$ and ensure the existence of an infinite computable
antichain by making the $x_n$'s pairwise incomparable.

We further make sure that, for all $e \in \omega$, $P$ meets the requirement:
\[ R_e \colon (\exists y) \bigl( (\Phi_e (y) = 1 \implies (\forall^\infty z \in
P) z \preceq y) \land
         (\Phi_e (y) = 0 \implies (\forall^\infty z \in P) y \preceq z)
\bigr).\]
Here, as usual, $\Phi_e$ is the function computed by the Turing machine of
index $e$ and $\forall^\infty$ means \lq for all but finitely
many\rq.\smallskip

We first show that meeting all the requirements implies that $P$ satisfies
the statement of the Lemma. If $I$ is a computable initial interval of $P$
with characteristic function $\Phi_e$, fix $y$ given by $R_e$. We must have
$\Phi_e (y) \in \{0,1\}$. If $\Phi_e (y) =0$ then, by $R_e$, $(\forall^\infty
z \in P) y \preceq z$. As $y \notin I$, this implies that $I$ is finite and
hence $I= \Down I$ is the downward closure of a finite set. If $\Phi_e (y)
=1$, then by $R_e$ we have $(\forall^\infty z \in P) z \preceq y$. Thus $P
\setminus \down y$ and hence $I \setminus \down y$ are finite. As $y \in I$,
$I = \Down {(\{y\} \cup (I \setminus \down y))}$ is the downward closure of a
finite set.\smallskip

Our strategy for meeting a single requirement $R_e$ consists in fixing a
witness $y_n$ and waiting for a stage $s+1$ such that
\[ \Phi_{e,s} (y_n) \in \{0,1\}. \]
If this never happens, $R_e$ is satisfied. If $\Phi_{e,s} (y_n) = 0$, we put
every $x_m$ and $y_m$ with $m>s$ above $y_n$. If $\Phi_{e,s} (y_n) = 1$, we
put every $x_m$ and $y_m$ with $m>s$ below $y_n$. In this way $R_e$ is
obviously satisfied.\smallskip

To meet all the requirements, the priority order is $R_0, R_1, R_2, \ldots$.
At every stage $s$, we define a witness for $R_e$ via an index $n_{e,s}$ and
mark the requirements by a $\{0,1\}$-valued function $r(e,s)$ such that
$r(e,s)=0$ if and only if $R_e$ might require attention at stage
$s$.\smallskip

\textbf{Stage $\mathbf{s=0}$.} For all $e$, $n_{e,0}=e$ and
$r(e,0)=0$.\smallskip

\textbf{Stage $\mathbf{s+1}$.} We say that $R_e$ \emph{requires attention} at
stage $s+1$ if $e \leq s$, $n_{e,s} \leq s$, $r(e,s)=0$ and $\Phi_{e,s}
(y_{n_{e,s}}) \in \{0,1\}$. If no $R_e$ requires attention, then let
$n_{i,s+1} = n_{i,s}$ and $r(i,s+1) = r(i,s)$ for all $i$. Otherwise, let $e$
be least such that $R_e$ requires attention. Then $R_e$ \emph{receives
attention} at stage $s+1$ and $n=n_{e,s}$ is \emph{activated} and declared
\emph{low} if $\Phi_{e,s} (y_n) = 0$, \emph{high} if $\Phi_{e,s} (y_n) = 1$.
Let $n_{e,s+1} = n_{e,s}$ and $r(e,s+1)=1$. For $i<e$, $n_{i,s+1} = n_{i,s}$
and $r(i,s+1) = r(i,s)$. For $i>e$, $n_{i,s+1} = s+i-e$ and
$r(i,s+1)=0$.\smallskip

The following two properties are easily seen to hold:
\begin{enumerate}
 \item every $n$ is activated at most once;
 \item if $n$ is activated at stage $s$, then no $m$ such that $n<m<s$ is
     activated after $s$.
\end{enumerate}

We define $\preceq$ by stipulating that for all $n<m$:
\begin{enumerate}[\quad(i)]
 \item $x_n$ is incomparable with each of $y_n$, $x_m$ and $y_m$;
 \item $y_n \preceq$ ($\succeq$) $x_m, y_m$ if and only if $n$ is activated
     at some stage $s$ such that $n< s \leq m$, is declared low (high) and
     no $k<n$ is activated at any stage $t$ such that $s<t \leq m$.
\end{enumerate}

When (ii) occurs, it follows by (2) that no $k<n$ is activated at any stage
$t$ such that $n<t\leq m$.

\begin{nclaim}
$P$ is a partial order.
\end{nclaim}
\begin{proof}[Proof of claim]
We use $z_n$ to denote one of $x_n$ and $y_n$.

To show antisymmetry, suppose for a contradiction that $z_n \preceq z_m$ and
$z_m \preceq z_n$ with $n<m$. By (i) $z_n$ must be $y_n$. Since $n$ can be
activated only once, it follows that $n$ is activated at some stage $s$ with
$n<s \leq m$ and, by (ii), is declared both low and high, a contradiction.

To check transitivity, let $z_n\prec z_m\prec z_p$. Notice that $n$, $m$ and
$p$ are all distinct. We consider the following cases:
\begin{enumerate}[(a)]
\item $n<m,p$. Then $z_n=y_n$ and $n$ is activated and declared low at some
    stage $s$ such that $n<s \leq m$. It is easy to verify that no $k<n$ is
    activated at any stage $t$ such that $n<t \leq p$, and thus $y_n
    \preceq z_p$.
\item $m<n,p$. Then $z_m=y_m$ and $m$ is declared both high and low,
    contradiction.
\item $p<n,m$. Then $z_p=y_p$ and $p$ is activated and declared high at
    some stage $s$ such that $p<s\leq m$. As in case (a), it is easy to
    check that no $k<p$ is activated at any stage $t$ such that $p<t \leq
    n$, and so $z_n \preceq y_p$.\qedhere
\end{enumerate}
\end{proof}

\begin{nclaim}\label{claim2}
Every $R_e$ receives attention at most finitely often and is satisfied.
\end{nclaim}
\begin{proof}[Proof of claim]
As usual, the proof is by induction on $e$. Let $s$ be the least such that no
$R_i$ with $i<e$ receives attention after $s$. Let $n = n_{e,s}$. Then $n =
n_{e,t}$ for all $t \geq s$, because a witness for a requirement changes only
when a stronger priority requirement receives attention. Similarly, $r(e,t)
=0$ for all $t \geq s$ such that $R_e$ has not received attention at any
stage between $s$ and $t$. If $\Phi_e (y_n) \notin \{0,1\}$, $R_e$ is clearly
satisfied. Suppose that $\Phi_e (y_n) = 0$ (case $1$ is similar) and let $t$
be minimal such that $t \geq \max \{s,e,n\}$ and $\Phi_{e,t}(y_n)=0$. Then
$R_e$ receives attention at stage $t+1$, $n$ is activated and declared low
and no $m<n$ will be activated after stage $t+1$ (because $n_{i,u}>n$ for all
$i>e$ and $u>t$). Then $y_n \preceq x_m, y_m$ for all $m>t$ and so $R_e$ is
satisfied.
\end{proof}

Claim \ref{claim2} completes the proof of the Lemma.
\end{proof}

\section{Open problems}\label{section:problems}
The results of Sections \ref{section:WKL} and \ref{section:REC} leave open
the status of the right to left directions of Theorems \ref{Bonnet1} and
\ref{Bonnet3}. Each of the statements (1) \lq\lq every partial order with an
infinite antichain contains an initial interval which is not a finite union
of ideals\rq\rq\ and (2) \lq\lq every partial order with an infinite
antichain has uncountably many initial intervals\rq\rq\ can be either
equivalent to \WKL\ or of strength strictly between \RCA\ and \WKL.

The latter case would be quite interesting, since the only mathematical
statements with this intermediate strength are those from measure theory that
are equivalent to the system \WWKL. Bienvenu, Patey, and Shafer improved
Theorems \ref{Bonnet1<-nRCA} and \ref{Bonnet3<-nRCA} by showing that \WWKL\
does not imply neither (1) nor (2). These results are obtained by modifying
the proof of Lemma \ref{REC}. The draft \cite{BPS} includes also other
non-implications involving statements (1) (called \NCF\ there) and (2).

On the other hand, Gregory Igusa (in private communications) claims that
there cannot be a uniform proof of \WKL\ from (1). This claim does not rule
out the possibility that (1) implies \WKL: e.g.\ there might exist a proof
using twice the statement, the second time using it on a partial order built
from the initial interval obtained by the first application.

\bibliographystyle{alpha}
\bibliography{initialintervals}{}

\end{document}